\documentclass[a4paper]{article}

\usepackage{amsfonts,amsmath,amsthm}
\usepackage{amssymb}
\usepackage{mathrsfs}

\theoremstyle{plain}
\newtheorem{thm}{Theorem}[section]
\newtheorem{lem}[thm]{Lemma}

\theoremstyle{definition}

\newtheorem{ex}[thm]{Example}
\newtheorem{rmk}[thm]{Remark}
\newtheorem*{ack}{Acknowledgement}

\numberwithin{equation}{section}

\newcommand{\R}{\mathbb{R}}
\newcommand{\Z}{\mathbb{Z}}

\newcommand{\K}{\rm{K}}
\newcommand{\W}{\rm{W}}
\newcommand{\kl}{\rm{(k)}}
\newcommand{\w}{\rm{(w)}}
\newcommand{\kkw}{\rm{(kkw)}}
\newcommand{\kww}{\rm{(kww)}}
\newcommand{\red}{\rm{red}}
\newcommand{\h}{\rm{high}}
\newcommand{\cubic}{\rm{c}}
\newcommand{\T}{\rm{T}}
\newcommand{\Op}{\rm{Op}}
\newcommand{\vp}{v^+}

\newcommand{\supp}{\mathrm{supp \, }}

\newcommand{\pa}{\partial}
\newcommand{\eps}{\varepsilon}
\newcommand{\tl}[1]{\widetilde{#1}}
\newcommand{\jb}[1]{\langle #1 \rangle}
\newcommand{\ab}[1]{\left| #1 \right|}
\newcommand{\no}[1]{\left\| #1 \right\|}
\newcommand{\op}[1]{\mathcal{#1}}

\title{Global existence for systems of nonlinear wave and Klein-Gordon equations in two space dimensions under a kind of the weak null condition}
\author{Minggang CHENG
		 \thanks{
              Department of Mathematics, Graduate School of Science, 
			Osaka University, Toyonaka, Osaka 560-0043, Japan. 
              Email: {\tt u718265c@ecs.osaka-u.ac.jp} 
             }
           }

\begin{document}
\date{}
\maketitle

%%%%%%%%%%%%%%%%%%%%%%%%%%%%%%%%%%%%%%%%%%%%%%%%%%%%%%%%%%%%%%%%%%%%%%%%%%%%%%%%%%%%%%%%%%%%%%%%%%%%%%%%%%%%%%%%%%%%%%%%%%%%%%%%%%%%%%%%%%%%%%%
%%%%%%%%%%%%%%%%%%%%%%%%%%%%%%%%%%%%%%%%%%%%%%%%%%%%%%%%%%%%%%%%%%%%%%%%%%%%%%%%%%%%%%%%%%%%%%%%%%%%%%%%%%%%%%%%%%%%%%%%%%%%%%%%%%%%%%%%%%%%%%%
\frenchspacing

\begin{abstract}

We consider the coupled systems of nonlinear wave and Klein-Gordon equations in two space dimensions with cubic nonlinearity.
For this kind of systems, the small data global existence is already known if the cubic nonlinearity satisfies a certain condition related to the null condition.
In this article, our aim is to investigate the small data global existence under a condition related to the weak null condition.
We also make a remark on the asymptotic behavior of global solutions.

\end{abstract}

\section{Introduction}

We investigate the Cauchy problem for coupled systems of nonlinear wave and Klein-Gordon equations in two space dimensions.

Let $u=(u_j)_{1 \le j \le N}$ be an $\R^{N}$-valued unknown function of $(t,x) \in (0, \infty) \times \R^2$.
In what follows, we always assume 
\begin{align}
	m_{j}>0 ~{\rm for}~ 1 \le j \le N_{0} ~~{\rm and}~~m_{j}=0 ~{\rm for}~ N_{0}+1 \le j \le N 
	\label{eq;mj}
\end{align}
with some $N_0 \in \{0,1,\ldots, N\}$, and we write $u=(v, w)$ with
\begin{align*}
	v=(v_{j})_{1 \leq j \leq N_{0}}=(u_{j})_{1 \leq j \leq N_{0}}, \ w=(w_{j})_{N_{0}+1 \leq j \leq  N}=(u_{j})_{N_{0}+1 \leq j \leq N},
\end{align*}
where we neglect the first term condition in \eqref{eq;mj} and $v$ (resp. the second term condition in \eqref{eq;mj} and $w$) if $N_0=0$ (resp. $N_0=N$).
We call $v_j$ as the Klein-Gordon component, and $w_j$ as the wave component.
Throughout this article, we write $\pa_0:=\pa_t=\pa/\pa t$, $\pa_{k}:=\pa/\pa x_{k}$ for $k=1, 2$.
The d'Alembertian is written as $\Box = \pa_{t}^{2} - \Delta$.
In this article, we focus on the following type of coupled system
\begin{align}
	(\Box+m_{j}^2) u_{j} = F_{j}(v, \pa u) \ \ \textrm{in} \ (0, \infty) \times \R^2 \ \textrm{for} \ j = 1, 2,\ldots,N, 
	\label{eq;nkgw}
\end{align}
where $\pa u=(\pa_a u)_{0 \le a \le 2}$.
We prescribe the initial condition by 
\begin{align}
	u_j (0,x)=\eps f_j(x), \ (\pa_t u_j)(0,x)= \eps g_j(x), \ \ \ x \in \R^{2}
	\label{eq;id}
\end{align}
for $ j = 1, 2,\ldots,N$, where we suppose that $f_j$, $g_j \in C^{\infty}_{0}(\R^2)$, and $\eps$ is a small positive parameter.
We concentrate on this simple case where $F$ depends only on $v$ and $\pa u=(\pa v, \pa w)$, because the difference of the behaviors of $w$ and $\pa w$ makes it complicated to consider the case where $F$ depends also on $w$ itself \footnote{For example, the readers can see the difference in Lemma \ref{lem;hom} and \ref{lem;inhom} below; 
moreover the energy estimate does not give a natural control of the $L^{2}$-norm of $w$.}.
We assume that $F=(F_{j})_{1 \le j \le N}$ is smooth in its arguments and satisfies  
\begin{align}
	F(v,\pa u)=\op{O}(|v|^{3}+|\pa u|^{3})
	\label{eq;nbh}
\end{align}
near $(v, \pa u)=(0, 0)$.

Let us recall some previous researches briefly.
For systems of nonlinear Klein-Gordon equations (namely the case where $N_0=N$ and our system is $(\Box+m_j^2)v_j=F_j(v, \pa v)$), it is a known fact that the small data global existence holds for the Cauchy problem \eqref{eq;nkgw}-\eqref{eq;id} without any further assumption, as the power of the nonlinearity $F$ is super-critical (see Ozawa-Tsutaya-Tsutsumi~\cite{OTT} and  Simon-Taflin~\cite{ST}).

On the other hand, for systems of nonlinear wave equations (namely the case where $N_0=0$ and our system is $\Box w=F(\pa w)$), the situation is different from the above case in that the solution can blow up in finite time for some cubic nonlinearity $F$ even if the initial data are small enough.
Thus we need some additional condition to obtain the global existence.
For example, it is known that we have global existence results for our system, provided that the cubic part $F^{\cubic}$ of the nonlinearity $F$ satisfies the (cubic) null condition
\begin{align}
	F^{\cubic}\bigl((\omega_a Y)_{0 \le a \le 2} \bigr)= 0, \qquad \omega=(\omega_1, \omega_2) \in \mathbb{S}^{1}, \ Y\in \R^{N}
	\label{condi;null}
\end{align}
with $\omega_0=-1$, where the left-hand side means that $\omega_a Y$ is substituted in place of $\pa_a w$ and $\mathbb{S}^1$ denotes the unit circle (see Godin~\cite{God} and Hoshiga~\cite{H} for the above result; see also Katayama~\cite{Kata1,Kata2} for the case where $F=F(w, \pa w)$).
The null condition was originally introduced by Klainerman~\cite{Kla2} and Christodoulou~\cite{Christ} to ensure the small data global existence for three space dimensional wave equations with quadratic nonlinearities.

After Lindblad-Rodnianski~\cite{LinRod} introduced the weak null condition, which has not yet been proved to be sufficient for the small data global existence, some related sufficient conditions weaker than the null condition have been intensively researched for two and three space dimensional wave equations.
Among such weak null conditions, we would like to focus on the one introduced by Katayama-Matsumura-Sunagawa~\cite{KMatsuS} in two space dimensions, which we call the KMS condition:
\begin{description}
	\item[(\rm{KMS})]
	There is $\op{J}=\op{J}(\omega) \in C(\mathbb{S}^{1}; \op{S}^{N}_{+})$ such that 
		\begin{align}
			\label{condi;KMS}
			Y^{\T} \op{J}(\omega)F^{\cubic}\bigl((\omega_a Y)_{0 \le a \le 2}\bigr) \ge 0, \qquad \omega=(\omega_1, \omega_2) \in \mathbb{S}^{1}, \ Y\in \R^{N}
		\end{align}
	with $\omega_0=-1$, where $\op{S}^{N}_{+}$ is the set of real symmetric positive-definite matrices of size $N \times N$.
\end{description}
Here and hereafter, $\op{A}^{\T}$ denotes the transpose of a matrix $\op{A}$.
As usual, $\R^{N}$-vectors are identified with $N \times 1 $-matrices.
It is trivial to see that the KMS condition \eqref{condi;KMS} is weaker than the null condition \eqref{condi;null};
however the KMS condition ensures the small data global existence.

Finally, for systems of nonlinear wave and Klein-Gordon equations (namely the case where $1 \le N_0 < N$), we also need some restriction on the nonlinearity to obtain the small data global existence as these systems contain nonlinear wave equations.
Motivated by the previous works \cite{Geo1,Kata3,Kata4,LFM} for systems with quadratic nonlinearity in three space dimensions, Aiguchi~\cite{Aig} investigated the Cauchy problem for \eqref{eq;nkgw}-\eqref{eq;id}, and proved the small data global existence, assuming that the interaction between wave components in the wave equations satisfies the null condition (recently systems with quadratic nonlinearities in two space dimensions are also widely studied; see for example, Dong~\cite{Don1,Don2}, Duan-Ma~\cite{DM}, Dong-Wyatt~\cite{DW} and Ma~\cite{Ma1,Ma2,Ma3,Ma4,Ma5}).

A natural question now arises: 
\textit{Can we obtain the small data global existence for \eqref{eq;nkgw}-\eqref{eq;id} if we replace the null condition in} \cite{Aig} \textit{by the} KMS \textit{condition?}
As far as we know, no previous research has investigated 
this question, and this is non-trivial because the scaling operator, used in the original work \cite{KMatsuS} for the wave equations, is not compatible with Klein-Gordon equations, so we need some modification of the treatment.

\begin{ack}
The author would like to express the deepest appreciation to Prof. Soichiro Katayama for his valuable advice, supports and encouragements. 
The author also thanks Prof. Hideaki Sunagawa for his useful suggestions and comments.
\end{ack}

%%%%%%%%%%%%%%%%%%%%%%%%%%%%%%%%%%%%%%%%%%%%%%%%%%%%%%%%%%%%%%%%%%%%%%%%%%%%%%%%%%%%%%%%%%%%%%%%%%%
%%%%%%%%%%%%%%%%%%%%%%%%%%%%%%%%%%%%%%%%%%%%%%%%%%%%%%%%%%%%%%%%%%%%%%%%%%%%%%%%%%%%%%%%%%%%%%%%%%%

\section{Condition and Main Theorem} \label{sec;main}
Before we proceed to the details, we give some notation.
We denote the cubic part of $F(v, \pa u)=\left(F_j(v, \pa u)\right)_{1 \le j \le N}$ as $F^{\cubic}(v, \pa u)=\left(F_j^{\cubic}(v, \pa u)\right)_{1 \le j \le N}$.
If we write $\pa^{\alpha}=\pa_{0}^{\alpha_0}\pa_{1}^{\alpha_1}\pa_{2}^{\alpha_2}$ for a multi-index $\alpha=(\alpha_0, \alpha_1, \alpha_2)$, then $F^{\cubic}$ has the form
\begin{align}
	F_{j}^{\cubic}(v, \pa u) 
	=\sum_{1 \le k \le l \le m \le N}\sum_{|\alpha|, |\beta|, |\gamma|\le 1}
	C_{jklm}^{\alpha\beta\gamma}
	(\pa^\alpha u_k)(\pa^\beta u_l)(\pa^\gamma u_m)
	\label{eq;Fcj}
\end{align}
with real constants $C_{jklm}^{\alpha\beta\gamma}$, where $C_{jklm}^{\alpha\beta\gamma}=0$,  
if $N_0+1\le k\le N$ and $|\alpha|=0$, or if $N_0+1\le l\le N$ and $|\beta|=0$, or if $N_0+1\le m \le N$ and $|\gamma|=0$, 
since $F$ is independent of $w$ itself.
For $1 \le j \le N$, we can divide $F^{\cubic}_{j}(v, \pa u)$ into the following four parts:
\begin{align*}
	F^{\kl}_{j}(v, \pa v) &:= \sum_{1 \le k\le l \le m \le N_{0}} \sum_{|\alpha|, |\beta|, |\gamma| \le 1} C_{jklm}^{\alpha\beta\gamma}(\pa^{\alpha}v_{k})(\pa^{\beta}v_{l})(\pa^{\gamma}v_{m}),\\ 
	F^{\kkw}_{j}(v, \pa u) &:= \sum_{\substack{1\le k \le l \le N_{0} \\ N_{0}+1\le m \le N}} \sum_{\substack{|\alpha|, |\beta|\le 1 \\ |\gamma|=1}} C_{jklm}^{\alpha\beta\gamma} \pa^{\alpha}v_{k})(\pa^{\beta}v_{l})(\pa^{\gamma}w_{m}),\\ 
	F^{\kww}_{j}(v, \pa u) &:= \sum_{\substack{1\le k \le N_{0} \\ N_{0}+1\le l \le m \le N}} \sum_{\substack{|\alpha| \le 1 \\ |\beta|, |\gamma|=1}} C_{jklm}^{\alpha\beta\gamma}(\pa^{\alpha}v_{k})(\pa^{\beta}w_{l})(\pa^{\gamma}w_{m}),\\
	F^{\w}_{j}(\pa w) &:= \sum_{N_{0}+1 \le k \le l \le m \le N} \sum_{|\alpha|,|\beta|,|\gamma|=1} C_{jklm}^{\alpha\beta\gamma}(\pa^{\alpha}w_{k})(\pa^{\beta}w_{l})(\pa^{\gamma}w_{m}).
\end{align*}
Since a multi-index $\alpha$ with $|\alpha|=1$ can be identified with a number $a \in \{0, 1, 2\}$ satisfying $\pa^{\alpha}=\pa_a$, we also write $F^{\w}_{j}(\pa w)$ as 
\begin{align*}
	F^{\w}_{j}(\pa w)=\sum_{N_0+1 \le k \le l \le m \le N}\sum^{2}_{a, b, c=0}C^{abc}_{jklm}(\pa_aw_k)(\pa_bw_l)(\pa_cw_m).
\end{align*}
We define $F_{j}^{\h}(v, \pa u) = F_{j}(v, \pa u) - F_{j}^{\cubic}(v, \pa u)$, so that $F_j^{\h}(v, \pa u)=\op{O}(|v|^4+|\pa u|^4)$ near $(v, \pa u)=(0,0)$.

We write $F=(F_{\K}, F_{\W})$ with 
\begin{align*}
               F_{\K}:=(F_{j})_{1 \le j \le N_{0}}, \; F_{\W}:=(F_{j})_{N_{0}+1 \le j \le N}.
\end{align*}
Similarly, for $\ast=\cubic, \kl, \kkw, \kww, \w, \h$, we write 
\begin{align*}
	F^{\ast}:=(F_{j}^{\ast})_{1 \leq j \leq N}, 
	\; F^{\ast}_{\K}:=(F_{j}^{\ast})_{1 \leq j \leq N_{0}}, 
	\;F^{\ast}_{\W}:=(F_{j}^{\ast})_{N_{0}+1 \leq j \leq N}.
\end{align*}

Notice that $F_{\W}^{\w}$ is what we call the \textit{interaction between the wave components in the wave equations} in the introduction.
Throughout this article, we always take
\begin{align*}
	\omega_0:=-1 \text{ and } N_1:=N-N_0.
\end{align*}
We define the \textit{reduced nonlinearity} $F^{\red}_{\W}=\left(F_j^{\red}\right)_{N_0+1 \le j \le N}$ of $F_{\W}^{\w}$ by
\begin{align} 
	\notag
	F_j^{\red}(\omega, Y)&=F_j^{\w}\big( (\omega_a Y)_{0 \le a \le 2} \big)\\
	\label{eq;wred}
	&=\sum_{N_0+1 \le k \le l \le m \le N}\sum^2_{a,b,c=0}C^{abc}_{jklm}\omega_a\omega_b\omega_cY_kY_lY_m
\end{align}
for $\omega=(\omega_1, \omega_2) \in \mathbb{S}^1$ and $Y=(Y_j)_{N_0+1 \le j \le N} \in \R^{N_1}$.

In this notation, the null condition for $F^{\w}_{\W}$ in \cite{Aig} can be written as follows:
\begin{align}
	F^{\red}_{\W}(\omega, Y)=0, \quad \omega \in \mathbb{S}^{1}, \ Y \in \R^{N_1}.
	\label{condi;kgwnull}
\end{align}
Now we would like to introduce our condition:
\begin{description}
	\item[({\bf W})]
	The KMS condition for $F^{\w}_{\W}$: 	There is $\op{J}=\op{J}(\omega) \in C(\mathbb{S}^{1}; \op{S}^{N_1}_{+})$ such that 
	\begin{align*}
		Y^{\T} \op{J}(\omega)F^{\red}_{\W}(\omega, Y) \ge 0, \qquad \omega \in \mathbb{S}^{1}, \ Y \in \R^{N_1},
	\end{align*}
	where, as before, $\op{S}^{N_1}_{+}$ is the set of real symmetric positive-definite matrices of size $N_1 \times N_1$.
\end{description}
The following theorem is our main result.

\begin{thm} \label{thm;main}
	Assume that the condition {\rm (W)} is satisfied. 
	Given $f$, $g \in C^{\infty}_{0}(\R^{2};$ $\R^{N})$, we can take a positive constant $\eps_0$ such that there is a global smooth solution $u=(v, w)$ in $[0, \infty) \times \R^2$ to the Cauchy problem \eqref{eq;nkgw}-\eqref{eq;id} for any $\eps \in (0, \eps_0]$.
\end{thm}

As was mentioned in the introduction, the null condition implies the KMS condition;
in other words, the condition \eqref{condi;kgwnull} implies the condition (W). 
Hence our theorem covers the previous result in \cite{Aig}.
The example below shows that our theorem is a real extension.

\begin{ex} \label{ex}
	For $u=(v, w)$ with $N=2$ and $N_0=N_1=1$, we consider a system
	\begin{align}
		\begin{cases}
			(\Box+m^2)v=F_1(v, \pa u), \\
			\Box w=F_2(v, \pa u), 
		\end{cases}
	\label{eq;ex}
	\end{align}
	where $m>0$ and 
	\begin{align*}
		F_2^{\w}(\pa w)=-c(\pa_a w)^2(\pa_t w)+(\pa_t w)\{(\pa_t w)^2-(\pa_1 w)^2-(\pa_2 w)^2\}
	\end{align*}
	for some $c \ge 0$ and $a=0,1,2$.
	We do not assume any additional assumptions on $F_1$, $F_2^{\ast}$ for $\ast=\kl, \kkw, \kww, \h$.
	As $F_{\W}^{\w}(\pa w)=F_2^{\w}(\pa w)$ and $Y=Y_2$, we have 
	\begin{align*}
		F_{\W}^{\red}(\omega, Y)=c\omega_a^2Y^3.
	\end{align*}
	The null condition \eqref{condi;null} for $F_2^{\w}$ is violated unless $c=0$;
	however the condition (W) is satisfied as we have
	\begin{align*}
		Y^{\T}F_{\W}^{\red}(\omega, Y)=c\omega_a^2Y^4 \ge 0, \quad \omega \in \mathbb{S}^{1}, \ Y(=Y_2) \in \R
	\end{align*}
	for any $c \ge 0$.
\end{ex}

This paper is organized as follows:
In Section \ref{section;preliminaries}, we give some preliminaries.
Section \ref{section;profile} is devoted to the profile system related to the KMS condition, and we introduce a technical transformation for the Klein-Gordon components in Section \ref{section;transform}.
Theorem \ref{thm;main} will be proved in Section \ref{section;sdge}.
In Section \ref{section;AB}, we briefly discuss the asymptotic behavior of global solutions.

%%%%%%%%%%%%%%%%%%%%%%%%%%%%%%%%%%%%%%%%%%%%%%%%%%%%%%%%%%%%%%%%%%%%%%%%%%%%%%%%%%%%%%%%%%%%%%%%%%%
%%%%%%%%%%%%%%%%%%%%%%%%%%%%%%%%%%%%%%%%%%%%%%%%%%%%%%%%%%%%%%%%%%%%%%%%%%%%%%%%%%%%%%%%%%%%%%%%%%%

\section{Preliminaries} \label{section;preliminaries}
In this section, we describe known decay estimates for the solutions to the linear wave and Klein-Gordon equations.
Throughout this article, we designate positive constants by $C$, whose value may change line by line.
We write $\jb{y}=\sqrt{1+|y|^2}$ for $ y \in \R^{d}$ with some natural number $d$. 
We denote by $\Z_+$ the set of non-negative integer.

To begin with, we introduce vector fields:
\begin{align*}
	L_{1}:=x_{1}\pa_{t}+t\pa_{1}, \quad L_{2}:=x_{2}\pa_{t}+t\pa_{2}, \quad \Omega:=x_{1}\pa_{2}-x_{2}\pa_{1},
\end{align*}
which are introduced by Klainerman~\cite{Kla1, Kla2}.
We define 
\begin{align*}
	\Gamma=(\Gamma_1, \Gamma_2,...,\Gamma_6) := (L_{1}, L_{2}, \Omega, \pa_0, \pa_1, \pa_2).
\end{align*}
Using a multi-index $\alpha=(\alpha_1,\ldots,\alpha_6) \in \Z^6_+$, we write $\Gamma^\alpha=\Gamma^{\alpha_1}\Gamma^{\alpha_2}\cdot\cdot\cdot\Gamma^{\alpha_6}$.
For a smooth function $\phi=\phi(t,x)$ and a non-negative integer $s$, we put
\begin{align*}
	|\phi(t,x)|_{s}=\sum_{|\alpha| \leq s}|\Gamma^{\alpha}\phi(t,x)|, \quad \|\phi(t)\|_{s}=\bigl\| |\phi(t,\cdot)|_{s}| \bigr\|_{L^{2}(\R^{2})}.
\end{align*}

We introduce the commutator $[A,B]=AB-BA$ between operators $A$ and $B$.
By simple computation, for $j=1,2$, $a=0,1,2$ and $m \ge 0$, we get 
\begin{align*}
[\Box+m^2, L_j]=[\Box+m^2, \Omega]=[\Box+m^{2},\pa_{a}]=0.
\end{align*}
Therefore for any multi-index $\alpha$ and a smooth function $\phi$, we get 
\begin{align}
	(\Box+m^{2})(\Gamma^{\alpha}\phi)=\Gamma^{\alpha} \bigl( (\Box+m^{2})\phi \bigr).
	\label{eq;comm}
\end{align}
We can also check that $[\Gamma_j, \Gamma_k]$ can be written as a linear combination of vector fields in $\Gamma$; 
in other words, if we write symbolically, we have $[\Gamma, \Gamma]=\Gamma$, which implies $|\Gamma^\alpha \phi|_s \le C|\phi|_{s+|\alpha|}$.
Especially, we have $[\Gamma, \pa]=\pa$.
Hence for any non-negative integer $s$, there is a positive constant $C_s$ such that we have
\begin{align}
	\frac{1}{C_{s}}|\pa \phi(t,x)|_{s} \leq \sum_{|\alpha| \leq s} \big| \pa(\Gamma^{\alpha}\phi)(t,x) \big| \leq C_{s}|\pa \phi(t,x)|_{s}
	\label{eq;Cs}
\end{align}
for any smooth function $\phi$.

We firstly describe a known decay result for solutions to linear Klein-Gordon equations
\begin{align}
	(\Box+m^{2})v(t,x)=\Phi(t,x), \qquad (t,x)\in (0, \infty) \times \R^{2}
	\label{eq;KG}
\end{align}
with $m > 0$.
We use the decay estimate in Georgiev~\cite{Geo2};
however, since we are working in the situation of compactly supported data, the statement can be simplified as follows.
\begin{lem} \label{lem;v}
	Suppose that $0 < T \le \infty$ and $R>0$.
	Let $v$ be a smooth solution to \eqref{eq;KG}.
	We assume that $v(0, \cdot)$ has compact support, and that $\Phi(t,x)=0$ for all $(t,x) \in [0, T) \times \R^2$ satisfying $|x| \ge t+R$.
	For $\rho \ne 1$, there exists a positive constant $C=C(m, \rho, R)$, independent of $T$, such that
	\begin{align}
		\jb{t+|x|}|v(t,x)| \le C \|v(0)\|_5 +Ct^{(1-\rho)_+} \sup_{\tau \in [0,t]} \jb{\tau}^{\rho} \|\Phi(\tau)\|_4
		\label{eq;v}
	\end{align}
	for $(t,x) \in (0, T) \times \R^2$, where $a_+=\max\{a, 0\}$ for $a \in \R$.
\end{lem}

\begin{proof}
	Let $\{ \chi_j \}_{j \in \Z_{+}} \bigl( \subset C^\infty_0(\R) \bigr)$ be a partition of unity in $[0, \infty)$ satisfying $\supp \chi_j \subset [2^{j-1}, 2^{j+1}]$ for $j \ge 1$ and $\supp \chi_0 \cap [0,\infty) \subset [0,2)$.
	Then it is proved in \cite{Geo2} that 
	\begin{align*}
		\jb{t+|x|}|v(t,x)| &\le C\sum^{\infty}_{j=0} \sum_{|\alpha| \leq 5} \left\| \jb{\cdot} \chi_{j}(|\cdot|) \Gamma^{\alpha} v(0,\cdot) \right\|_{L^{2}}\\
		&\quad +C\sum^{\infty}_{j=0} \sum_{|\alpha|\leq 4} \sup_{\tau \in [0,t]} 
		\chi_{j}(\tau) \big\| \jb{\tau+|\cdot|} \Gamma^{\alpha} \Phi(\tau,\cdot) \big\|_{L^{2}} \\
	\end{align*}
	without any support condition on $v(0)$ and $\Phi$.
	It is easy to see that the first term on the right-hand side is bounded by $C\|v(0)\|_5$ if $v(0)$ is compactly supported.
	Let $\|\Phi(t)\|_4 \le M_0 \jb{t}^{-\rho}$.
	Then, in view of the support condition for $\Phi$, we see that the second term is bounded by
	\begin{align*}
		CM_0\sum^{\infty}_{j=0} \sup_{\tau \in [0,t]} \chi_j(\tau) \jb{\tau}^{1-\rho} \le CM_0 \jb{t}^{(1-\rho)_+}
	\end{align*}
	because we have $\chi_j(\tau)\jb{\tau}^{1-\rho} \le C2^{j(1-\rho)}$.
\end{proof}

We switch to known decay results for solutions to linear wave equations
\begin{align}
	\begin{split}
		&\Box w(t,x)=\Psi(t,x), \qquad \qquad \qquad \qquad \qquad (t,x)\in (0, \infty) \times \R^{2}, \\
		&w(0,x)=w_{0}(x), \ (\pa_t w)(0,x)=w_{1}(x), \qquad x \in \R^2.
	\end{split}
	\label{eq;LW}
\end{align}
We use the weighted $L^{\infty}$-$L^{\infty}$ estimates in Hoshiga-Kubo~\cite{HK} and Kubo~\cite{Kubo}.
To state the estimates, we introduce
\begin{align} \label{eq;W-}
	W_{-}(t,r)&:=\min\{\jb{r},\jb{t-r} \}.
\end{align}
The first pair consists of the decay estimates of solutions and its derivatives to the homogeneous wave equation (namely the case where $\Psi=0$ in $(0, \infty) \times \R^{2}$).
\begin{lem} \label{lem;hom}
	Let $w$ be a smooth solution to \eqref{eq;LW} with $\Psi=0$.
	Suppose that $\mu>0$.
	For any $\bigl( w_{0}, w_{1} \bigr) \in C^{\infty}(\R^2) \times C^{\infty}(\R^2)$, it holds that 
	\begin{align}
		\begin{array}{l}
			\jb{t+|x|}^{1/2}\jb{t-|x|}^{1/2}  |w(t,x)| \le C\op{A}_{2+\mu, 0}[w_{0}, w_{1}], \\
			\jb{t+|x|}^{1/2}\jb{t-|x|}^{3/2}  |\pa w(t,x)| \le C\op{A}_{3+\mu, 1}[w_{0}, w_{1}]
		\end{array}
	\end{align}
	for $(t,x) \in [0,T) \times \R^2$, where we put
	\begin{align*}
		\op{A}_{\rho, s}[w_{0}, w_{1}]:= \sum_{|\alpha|+|\beta| \le s} \left( \sum_{|\gamma| \le 1} \| \jb{\cdot}^{\rho}\pa_{x}^{\alpha+\gamma}\Omega^{\beta}w_{0} \|_{L^{\infty}}+\|\jb{\cdot}^{\rho}\pa_{x}^{\alpha}\Omega^{\beta}w_{1}\|_{L^{\infty}} \right).
	\end{align*}
\end{lem}

The second pair is for the inhomogeneous wave equation (namely the case where $w_{0}(x)=w_{1}(x)=0$ for all $x \in \R^2$).
\begin{lem} \it \label{lem;inhom}
	Let $w$ be a smooth solution to \eqref{eq;LW} with initial data $w_{0}=w_{1}=0$.
	Let $s$ be a non-negative integer.
	Suppose that $\xi \ge 0$, $0<\zeta<1/2$ and $\eta>0$.
	Then there exists a positive constant $C=C(\xi, \zeta, \eta)$ such that 
	\begin{align}
		\label{eq;west}
		&\jb{t+|x|}^{1/2-\xi} \jb{t-|x|}^{\zeta} |w(t,x)| \le C\op{B}_{\zeta-\xi, \eta,0}[\Psi](t,x), \\
		\label{eq;pawest}
		&\jb{t+|x|}^{-\xi} \jb{x}^{1/2}	\jb{t-|x|}^{1+\zeta}	|\pa w(t,x)| \le C\op{B}_{\zeta+\eta-\xi,0,1}[\Psi](t,x)
	\end{align}
	for $(t,x) \in [0,T) \times \R^2$, where
	\begin{align*}
		\op{B}_{\kappa, \mu, s}[\Phi](t,x)&=\sup_{0 \le \tau <t} \sup_{|y-x| \le t-\tau} \jb{y}^{1/2}\jb{\tau+|y|}^{1+\kappa} W_{-}(\tau, |y|)^{1+\mu} \\
		&\qquad \qquad \qquad \qquad \qquad \qquad \times \sum_{|\alpha|+|\beta| \le s}|\pa^{\alpha}\Omega^{\beta}\Phi(\tau, y)|.
	\end{align*}
\end{lem}

\begin{rmk}
In \cite{Kubo}, the above lemma is proved for more general weight $W_{-}$, but only for $\xi=0$.
However, using the fact that $\jb{\tau+|y|}^{\xi} \le \jb{t+|x|}^{\xi}$ for $(\tau, y)$ satisfying $0 \le \tau < t$ and $|y-x| \le t-\tau$, we can easily show the case where $\xi > 0$.
\end{rmk}

The following Sobolev type inequality is from Klainerman~\cite{Kla3}:
\begin{lem} \label{lem;xphix}
There exists a positive constant C such that 
\begin{align}
\sup_{x \in \R^2}\jb{x}^{1/2}|\varphi(x)| \le C \sum_{|\alpha|+|\beta| \le 2} \|\pa_x^\alpha \Omega^\beta \varphi \|_{L^2(\R^2)}
\label{xphix}
\end{align}
for any $\varphi \in C^{\infty}_{0}(\R^2)$.
\end{lem}

%%%%%%%%%%%%%%%%%%%%%%%%%%%%%%%%%%%%%%%%%%%%%%%%%%%%%%%%%%%%%%%%%%%%%%%%%%%%%%%%%%%%%%%%%%%%%%%%%%
%%%%%%%%%%%%%%%%%%%%%%%%%%%%%%%%%%%%%%%%%%%%%%%%%%%%%%%%%%%%%%%%%%%%%%%%%%%%%%%%%%%%%%%%%%%%%%%%%%

\section{The profile system for wave components} \label{section;profile}
For $x \in \R^2$, we use the polar coordinates $r=|x|$ and $\omega=|x|^{-1}x$ in the sequel.
We write
\begin{align*}
	\hat{\omega}:=(\omega_0, \omega_1, \omega_2)=(-1, \omega_1, \omega_2).
\end{align*}

We introduce
\begin{align*}
	\pa_{\pm}:=\pa_t \pm \pa_r \  \text{and} \ D_\pm:=(\pa_r \pm \pa_t)/2,
\end{align*}
where $\pa_r=\pa/\pa r=\sum^{2}_{j=1}(x_j/|x|)\pa_j$.
One of the important points in the method of the previous work \cite{KMatsuS} is to approximate $r^{1/2}\pa w(t,x)$ by $\hat{\omega}D_{-}\bigl(r^{1/2}w(t,x) \bigr)$;
however the scaling operator $S=t\pa_t+r\pa_r$, which is not compatible with Klein-Gordon equations, was used to estimate the error between them.
To recover the estimate without using the scaling operator $S$, we define
\begin{align}
	[\phi(t,x)]_s:=|\phi(t,x)|_s+\jb{t-r}|\pa \phi(t,x)|_{s-1}
\end{align}
for a natural number $s$ and a smooth function $\phi=\phi(t,x)$.

\begin{lem} \label{lem;profile}
	There exists a universal positive constant $C$ such that	
	\begin{align*}
		\left| r^{1/2} \pa \phi(t,x)-\hat{\omega} D_{-} \left( r^{1/2} \phi (t,x) \right) \right| \le C\jb{t+|x|}^{-1/2} [\phi(t,x)]_1, \ \ 1 \le \frac{t}{2} \le r
	\end{align*}
	holds for any smooth function $\phi=\phi(t,x)$.
\end{lem}

\begin{proof}
	Let $1 \le t/2 \le r$.
	Since we have $\pa_1=\omega_1\pa_r-r^{-1}\omega_2\Omega$ and $\pa_2=\omega_2\pa_r+r^{-1}\omega_1\Omega$, we obtain 
	\begin{align*}
		|r^{1/2}\pa_j \phi-\omega_j \pa_r (r^{1/2} \phi)| \le C \jb{t+r}^{-1/2}|\phi|_1, \ \ j=1,2.
	\end{align*}
	Therefore our task is to estimate $\pa_t(r^{1/2}\phi)$ and $\omega_j \pa_r (r^{1/2} \phi)$ for $j=1,2$.
	Observing that $\pa_t =-D_- +D_+=\omega_0 D_- +D_+$ and $\omega_j \pa_r=\omega_jD_- +\omega_jD_+$, it suffices to prove
	\begin{align}
		|D_+(r^{1/2}\phi)| \le C \jb{t+r}^{-1/2}[\phi]_1.
		\label{eq;D+est}
	\end{align}
	Since we have $D_+(r^{1/2}\phi)=r^{-1/2}\phi/4+r^{1/2}D_+ \phi$ and 
	\begin{align}
		D_+=\frac{1}{2(t+r)}\left( 2\sum_{k=1}^{2}\omega_kL_k+(t-r)(\pa_t-\pa_r) \right),
	\end{align}
	we get \eqref{eq;D+est} immediately.
\end{proof}

\begin{rmk}
	By an apparent modification of the above proof, we also have
	\begin{align}
		|\pa \phi(t,x)-\hat{\omega}D_{-}\phi(t,x)| \le C\jb{t+r}^{-1}[\phi(t,x)]_1, \ \ 1 \le \frac{t}{2} \le r. 
		\label{rmk;profile}
	\end{align}
\end{rmk}

For $R>0$, we set
\begin{align*}
	B_R=\{x \in \R^2; |x| \le R \}.
\end{align*}
Let $0 < T \le \infty$.
Suppose that $u=(v,w)$ is a solution to \eqref{eq;nkgw}-\eqref{eq;id} on $[0, T) \times \R^2$.
As $f, g \in C^{\infty}_{0}$, we can take a positive constant $R$ such that
\begin{align}
	\supp f \cup \supp g \subset B_{R},
	\label{eq;idfinite}
\end{align}
which implies
\begin{align}
	\supp u(t, \cdot) \subset B_{t+R}, \ \ 0 \le t <T
	\label{eq;finite}
\end{align}
by the finite propagation property.
Keeping this in mind, we define
\begin{align}
	\Lambda_{T,R}:=\left\{(t,x) \in [0,T) \times \R^2 ; 1 \le \frac{t}{2} \le |x| \le t+R\right\},
	\label{dfn;light cone}
\end{align}
and we write $\Lambda_R:=\Lambda_{\infty, R}$.
Note that the weights $\jb{t+r}^{-1}$, $\jb{r}^{-1}$, $\jb{t}^{-1}$, $r^{-1}$ and $t^{-1}$ are equivalent to each other in $\Lambda_{R}$, since we have
\begin{align}
	\jb{t+r}^{-1} \le \jb{r}^{-1} \le r^{-1} \le 2t^{-1} \le \sqrt{5}\jb{t}^{-1} \le \sqrt{5}(2+R)\jb{t+r}^{-1}, (t,x) \in \Lambda_R.
\end{align}
Writing the d'Alembertian in the polar coordinates, we obtain
\begin{align}
	r^{1/2}\Box\phi=\pa_+\pa_-(r^{1/2}\phi)-\frac{1}{4r^{3/2}}(4\Omega^2+1)\phi
	\label{eq;rbox}
\end{align}
for any smooth function $\phi=\phi(t,x)$.
We define $\op{W}=(\op{W}_j)_{N_0+1 \le j \le N}$ by
\begin{align}
	\op{W}(t,x):=D_-\bigl(r^{1/2}w(t,x)\bigr), \ \ (t,x) \in [0, T) \times (\R^2\backslash\{0\}).
	\label{eq;Wdfn}
\end{align}
Then \eqref{eq;rbox} leads to 
\begin{align} 
	\pa_+ \op{W}(t,x)=-\frac{1}{2t}F^{\red}_{\W}\bigl(\omega, \op{W}(t,x)\bigr)+H(t,x),
	\label{eq;profile}
\end{align}
where $H=H(t,x)$ is defined by
\begin{align}
	H=-\frac{1}{2}\left(r^{1/2}F_{\W}(v, \pa u)-\frac{1}{t}F^{\red}_{\W}(\omega, \op{W})\right)-\frac{1}{8r^{3/2}}(4\Omega^2+1)w.
\end{align}
We call \eqref{eq;profile} the \textit{profile system for wave components}, as $H$ can be treated as a reminder in $\Lambda_{T, R}$ (see Lemma \ref{lem;Hest} below).

For $\sigma \in \R$, we define
\begin{align}
	t_0(\sigma)=\max\{-2\sigma, 2\},
\end{align}
and
\begin{align*}
	\op{W}^{\ast}_{0}(\sigma, \omega)=\op{W}\bigl( t, (t+\sigma)\omega \bigr) \Big|_{t=t_0(\sigma)}.
\end{align*}
Observe that we have 
\begin{align}
	\jb{R}^{-1}\jb{\sigma} \le t_0(\sigma) \le 2\jb{\sigma}, \ \ \sigma \le R.
	\label{eq;t0sigma}
\end{align}
The following lemma explains how the profile system is used to obtain the decay estimate for $\pa w$ in $\Lambda_{T, R}$, and this is the only part where we use our condition (W).
This part is already proved implicitly in \cite{KMatsuS}, but we give a proof for its importance.

\begin{lem} \label{lem;rdw}
	Suppose that the condition {\rm (W)} is satisfied. Then we have
	\begin{align*}
		\jb{r}^{1/2}|\pa w(t,x)| &\le C \left( \ab{\op{W}^{\ast}_{0}(r-t, \omega)}+\int^{t}_{t_0(r-t)} \ab{H\bigl( \tau, (r-t+\tau)\omega \bigr)} d\tau \right) \\
		&\quad + C\jb{t+r}^{-1/2}[w(t,x)]_1
	\end{align*}
	for $(t,x)=(t,r\omega) \in \Lambda_{T,R}$, where $C$ is a positive constant independent of $T$.
\end{lem}

\begin{proof}
	In view of Lemma \ref{lem;profile}, it suffices to estimate $\op{W}(t,x)$.
	Let $\op{J}=\op{J}(\omega)$ be from the condition (W).
	As $\op{J} \in C(\mathbb{S}^{1}; \op{S}^{N_1}_{+})$, there is a positive constant $C$ such that
	\begin{align}
		\frac{1}{C}|Y|^2 \le Y^{\rm T}\op{J}(\omega)Y \le C|Y|^2, \ \ \omega \in \mathbb{S}^1, Y \in \R^{N_1}.
		\label{eq;YJYest}
	\end{align}
	Suppose that $(t,x) \in \Lambda_{T,R}$.
	Let $\sigma=r-t$ and we put
	\begin{align*}
		\op{W}^{\ast}(t, \sigma, \omega)=\op{W}\bigl(t, (t+\sigma)\omega\bigr),
	\end{align*}
	so that we have 
	\begin{align*}
		\pa_t\op{W}^{\ast}(t, \sigma, \omega)=(\pa_+ \op{W})\bigl(t, (t+\sigma)\omega\bigr)=-\frac{1}{2t}F^{\red}_{\W}\bigl(\omega, \op{W}^{\ast}(t, \sigma, \omega)\bigr)+H^{\ast}(t, \sigma, \omega),
	\end{align*}
	where $H^{\ast}(t, \sigma, \omega)=H\bigl(t, (t+\sigma)\omega\bigr)$.
	Recalling that $\op{J}$ is real symmetric, and using the condition (W), we obtain
	\begin{align*}
		\pa_t \bigl\{(\op{W}^{\ast}\bigr)^{\rm T} \op{J} \op{W}^{\ast} \bigl\} &=2(\op{W}^{\ast})^{\rm T} \op{J} \pa_t \op{W}^{\ast}=-\frac{1}{t}(\op{W}^{\ast})^{\rm T}\op{J}F_{\W}^{\red}(\omega, \op{W}^{\ast})+2(\op{W}^{\ast})^{\rm T} \op{J}H^{\ast} \\
		&\le 2(\op{W}^{\ast})^{\rm T} \op{J}H^{\ast} \le C \sqrt{(\op{W}^{\ast})^{\T}\op{J}\op{W}^{\ast}}\ab{H^{\ast}}.
	\end{align*}
	As we have $t \ge t_0(\sigma)$ in $\Lambda_{T,R}$, the above differential inequality implies
	\begin{align*}
		\ab{\op{W}^{\ast}(t, \sigma, \omega)} \le C \left( \ab{\op{W}^{\ast}_{0}(\sigma, \omega)}+\int^{t}_{t_0(\sigma)}|H^{\ast}(\tau, \sigma, \omega)|d\tau \right),
	\end{align*}
	which leads to the desired estimate for $\op{W}$ immediately. 
\end{proof}

We will play a similar game for $\Gamma^\alpha w$ with a multi-index $\alpha \in \Z^{6}_{+}$, but the condition (W) will not be used.
We define
\begin{align}
	\op{W}^{(\alpha)}(t,x):=D_{-}\bigl(r^{1/2}\Gamma^{\alpha}w(t,x)\bigr).
\end{align}
As $\Box(\Gamma^\alpha w)=\Gamma^\alpha \bigl(F(v, \pa u) \bigr)$, it follows from \eqref{eq;rbox} that 
\begin{align}
	\pa_+ \op{W}^{(\alpha)}=-\frac{1}{2t}\op{G}(\omega, \op{W})\op{W}^{(\alpha)}+H_{\alpha}
	\label{eq;D+Wa}
\end{align}
for $|\alpha| \ge 1$, where a matrix-valued function $\op{G}=(\op{G}_{jk})_{N_0+1 \le j,k \le N}$ and a vector-valued function $H_\alpha$ are given by
\begin{align*}
	\op{G}_{jk}(\omega, Y)=\frac{\pa F^{\red}_{j}}{\pa Y_k}(\omega, Y),
\end{align*}
and
\begin{align}
	H_\alpha(t,x)=-\frac{1}{2}\left(r^{1/2}\Gamma^{\alpha}\bigl(F_{\W}(v, \pa u)\bigr)-\frac{1}{t}\op{G}(\omega, \op{W})\op{W}^{(\alpha)}\right)-\frac{1}{8r^{3/2}}(4\Omega^2+1)\Gamma^{\alpha}w,
	\label{eq;Halpha}
\end{align}
respectively.
For a square matrix $\op{A}$, $\|\op{A}\|_{\Op}$ denotes its operator norm in the sequel.

\begin{lem} \label{lem;dGw}
	Let $|\alpha|=s \ge 1$.
	We define
	\begin{align*}
		\op{W}^{(\alpha)\ast}_{0}(\sigma, \omega):=\op{W}^{(\alpha)}\bigl(t, (t+\sigma)\omega\bigr)\Big|_{t=t_0(\sigma)}.
	\end{align*}
	If there is a positive constant $C_0$ such that
	\begin{align*}
		\left\| \op{G}\bigl(\omega, \op{W}(t,x)\bigr) \right\|_{\Op} \le 2C_0 \eps^2, \ \ (t,x) \in \Lambda_{T,R},
	\end{align*}
	then there is a positive constant $C$, which is independent of $T$, such that
	\begin{align*}
		\jb{r}^{1/2}|\pa \Gamma^{\alpha} w(t,x)| &\le C\left(\frac{t}{t_0(r-t)}\right)^{C_0 \eps^2} \ab{\op{W}^{(\alpha)\ast}_0(r-t, \omega)} \\
		&\quad +t^{C_0\eps^2} \int^{t}_{t_0(r-t)}\tau^{-C_0\eps^2}\ab{H_\alpha \bigl(t, (r-t+\tau)\omega \bigr)} d\tau \\
		&\quad +C\jb{t+r}^{-1/2}[w(t,x)]_{s+1}
	\end{align*}
	for $(t,x)=(t,r\omega) \in \Lambda_{T,R}$.
\end{lem}

\begin{proof}
	By Lemma \ref{lem;profile} again, we only have to estimate $\op{W}^{(\alpha)}(t,x)$.
	using the $\ast$-notation as in the proof of the previous lemma, from \eqref{eq;D+Wa} we get
	\begin{align*}
		\pa_t \ab{\op{W}^{(\alpha)\ast}}^2 \le \frac{2C_0\eps^2}{t}\ab{\op{W}^{(\alpha)\ast}}^2+2\ab{\op{W}^{(\alpha)\ast}} |H^{\ast}_{\alpha}|.
	\end{align*}
	Let $c>0$ and $\Phi_c=\sqrt{c+\left|\op{W}^{(\alpha)\ast}\right|^2}$.
	Then we get
	\begin{align*}
		\pa_t\Phi_c \le \frac{C_0\eps^2}{t}\Phi_c+|H^{\ast}_{\alpha}|,
	\end{align*}
	from which we get
	\begin{align*}
		\pa_t\left(t^{-C_0\eps^2}\Phi_c\right) \le t^{-C_0\eps^2}|H^{\ast}_{\alpha}|.
	\end{align*}
	Therefore we obtain
	\begin{align*}
		t^{-C_0\eps^2}\Phi_c(t,\sigma,\omega) \le t_0(\sigma)^{-C_0\eps^2}\Phi_c\bigl(t_0(\sigma), \sigma, \omega \bigr)+\int^{t}_{t_0(\sigma)}\tau^{-C_0\eps^2}|H^{\ast}_{\alpha}(\tau, \sigma, \omega)|d\tau.
	\end{align*}
	By taking the limit as $c \rightarrow +0$, we obtain the desired estimate for $\op{W}^{(\alpha)}$.
\end{proof}

Finally we give preliminary estimates for $H$ and $H_\alpha$.
\begin{lem} \label{lem;Hest}
	It holds that
	\begin{align}
		\notag
		|H(t,x)| &\le Ct^{-1/2}\jb{t-r}^{-2}[w]^3_1+Ct^{-3/2}|w|_2\\
		&\quad +Ct^{1/2}|(v, \pa u)|^2 \bigl(|v|_1+|(v, \pa u)|^2 \bigr)
		\label{eq;Hest}
	\end{align}
	for $(t,x) \in \Lambda_{T,R}$, where $C$ is a positive constant independent of $T$.
	If $|\alpha|=s \ge 1$, then there is a positive constant $C_s$, being independent of $T$, such that
	\begin{align}
		\notag
		|H_{\alpha}(t,x)| &\le C_s t^{1/2}|\pa w|^{3}_{s-1}+C_s t^{-1/2}\jb{t-r}^{-2}[w]^{3}_{s+1}+C_s t^{-3/2}|w|_{s+2} \\
		&\quad +C_s t^{1/2} |(v, \pa u)|^{2}_{s} \bigl(|v|_{s+1}+|(v, \pa u)|^{2}_{s} \bigr)
		\label{eq;Halphaest}
	\end{align}
	for $(t,x) \in \Lambda_{T,R}$.
\end{lem}

\begin{proof}
	Throughout this proof, we suppose that $(t,x)=(t,r\omega) \in \Lambda_{T,R}$.
	We start with the proof of \eqref{eq;Hest}.
	We have $r^{3/2}F^{\w}_{\W}(\pa w)=F^{\w}_{\W}(r^{1/2}\pa w)$ by the homogeneity of $F^{\w}_{\W}$.
	Hence $H$ can be written as 
	\begin{align*}
		H &= -\frac{1}{2t}\left( F^{\w}_{\W}(r^{1/2}\pa w)-F^{\red}_{\W}(\omega, \op{W}) \right)-r^{1/2}\frac{t-r}{2t}F^{\w}_{\W}(\pa w)\\
		&\qquad -\frac{1}{8r^{3/2}}(4\Omega^2+1)\omega-\frac{r^{1/2}}{2}\left(F_{\W}(v, \pa u)-F^{\w}_{\W}(\pa w) \right)\\
		&=: R_1+R_2+R_3+R_4.
	\end{align*}
	It is easy to see that
	\begin{align*}
		&|R_2| \le Ct^{-1/2}\jb{t-r}|\pa w|^3 \le Ct^{-1/2}\jb{t-r}^{-2}[w]^{3}_{1},\\
		&|R_3| \le Ct^{-3/2}|w|_2,
	\end{align*}
	since $t$ and $r$ are equivalent quantities in $\Lambda_{T,R}$.
	We also get
	\begin{align*}
		|R_4| &= \frac{r^{1/2}}{2} \ab{F^{\kl}_{\W}+F^{\kkw}_{\W}+F^{\kww}_{\W}+F^{\h}_{\W}} \\
		&\le Ct^{1/2} \left( \sum_{j=1}^{3}|(v, \pa v)|^{j}|\pa w|^{3-j}+|(v, \pa u)|^{4} \right) \\
		&\le Ct^{1/2} \bigl( |v|_{1}+|(v, \pa u)|^{2} \bigr)|(v, \pa u)|^2. 
	\end{align*}
	From the definition \eqref{eq;Wdfn} of $\op{W}$, we have
	\begin{align*}
		|\op{W}| \le r^{1/2}|D_- w|+\frac{|w|}{4r^{1/2}} \le Ct^{1/2}(|\pa w|+\jb{t+r}^{-1}|w|) \le Ct^{1/2}\jb{t-r}^{-1}[w]_1.
	\end{align*}
	It is easy to see that $r^{1/2}|\pa w| \le Ct^{1/2}\jb{t-r}^{-1}[w]_1$.
	Therefore it follows from Lemma \ref{lem;profile} that
	\begin{align}
		\notag
		|R_1| &\le Ct^{-1} \sum_{k,l,m}\sum_{a,b,c} \ab{(r^{1/2}\pa_aw_k)(r^{1/2}\pa_bw_l)(r^{1/2}\pa_cw_m)-(\omega_a\op{W}_k)(\omega_b\op{W}_l)(\omega_c\op{W}_m)} \\
		&\le Ct^{-1}(r^{1/2}|\pa w|+|\op{W}|)^2|r^{1/2}\pa w-\hat{\omega}\op{W}| \le Ct^{-1/2}\jb{t-r}^{-2}[w]^3_1.
	\end{align}

	Gathering the above estimates, we obtain \eqref{eq;Hest}.

	Next we turn to \eqref{eq;Halphaest} with $|\alpha|=s \ge 1$.
	Let $C^{abc}_{jklm}$ be from \eqref{eq;wred}.
	For functions $\phi=(\phi_k)$, $\psi=(\psi_k)$, $\eta=(\eta_k)$, with $k$ running from $N_0+1$ to $N$, we define $\op{C}_{\W}=(\op{C}_j)_{N_0+1 \le j \le N}$ by
	\begin{align*}
		\op{C}_j[\pa\phi, \pa\psi, \pa\eta]=\sum_{N_0+1 \le k \le l \le m \le N} \sum_{a,b,c=0}^{2}C^{abc}_{jklm}(\pa_a\phi_k)(\pa_b\psi_l)(\pa_c\eta_m),
	\end{align*}
	so that we have 
	\begin{align*}
		F^{\w}_{\W}(\pa w)=\op{C}_{\W}[\pa w, \pa w, \pa w].
	\end{align*}
	We also define $\op{C}^{\alpha}_{\W}(\pa w)=\bigl( \op{C}^{\alpha}_{j}(\pa w) \bigr)_{N_0+1 \le j \le N}$ by 
	\begin{align*}
		\op{C}^{\alpha}_{j}(\pa w)=\op{C}_{j}[\pa \Gamma^{\alpha}w, \pa w, \pa w]+\op{C}_{j}[\pa w, \pa\Gamma^{\alpha}w, \pa w]+\op{C}_{j}[\pa w, \pa w, \pa\Gamma^{\alpha}w]
	\end{align*}
	to write $H_\alpha$ as
	\begin{align*}
		H_\alpha &= -\frac{r^{1/2}}{2} \left( \Gamma^\alpha \left( F^{\w}_{\W}(\pa w) \right)-\op{C}^{\alpha}_{\W}(\pa w) \right)-\frac{1}{2t}\left(r^{3/2}\op{C}^{\alpha}_{\W}(\pa w)-\op{G}(\omega, \op{W})\op{W}^{(\alpha)}\right) \\
		&\quad -r^{1/2}\frac{t-r}{2t}\op{C}^{\alpha}_{\W}(\pa w)-\frac{1}{8r^{3/2}}(4\Omega^2+1)\Gamma^{\alpha}w \\
		&\quad -\frac{r^{1/2}}{2}\Gamma^{\alpha}\bigl(F_{\W}(v, \pa u)-F^{\w}_{\W}(\pa w) \bigr) =:R^{\alpha}_{0}+R^{\alpha}_{1}+R^{\alpha}_{2}+R^{\alpha}_{3}+R^{\alpha}_{4}.
	\end{align*} 
	Since we have $|\pa \Gamma^{\alpha}w-\Gamma^{\alpha}\pa w| \le C|\pa w|_{s-1}$, the Leibniz formula implies
	\begin{align*}
		|R^{\alpha}_{0}| \le Ct^{1/2}|\pa w|^{3}_{s-1}.
	\end{align*}
	If we use the notation 
	\begin{align*}
		\op{C}_{j}[\hat{\omega}X, \hat{\omega}Y, \hat{\omega}Z]=\sum_{N_0+1 \le k \le l \le m \le N} \sum^{2}_{a,b,c=0}C^{abc}_{jklm}(\omega_aX_k)(\omega_bY_l)(\omega_cZ_m)
	\end{align*}
	for $X,Y,Z \in \R^{N_1}$, we have $F^{\red}_{\W}(\omega, Y)=\op{C}_{\W}[\hat{\omega}Y, \hat{\omega}Y, \hat{\omega}Y]$.
	From this, we see that the $j$ th component of $\op{G}(\omega, \op{W})\op{W}^{(\alpha)}$ is 
	\begin{align*}
		\op{C}_{j}[\hat{\omega}\op{W}^{(\alpha)}, \hat{\omega}\op{W}, \hat{\omega}\op{W}]+\op{C}_{j}[\hat{\omega}\op{W}, \hat{\omega}\op{W}^{(\alpha)}, \hat{\omega}\op{W}]+\op{C}_{j}[\hat{\omega}\op{W}, \hat{\omega}\op{W}, \hat{\omega}\op{W}^{(\alpha)}].
	\end{align*}
	Therefore, recalling the definition of $\op{C}^{\alpha}_{\W}(\pa w)$ and going a similar way to the estimates of $R_1$, we obtain
	\begin{align*}
		|R^{\alpha}_{1}| &\le Ct^{-1}\bigl(r^{1/2}|\pa w|+|\op{W}| \bigr)^{2}|r^{1/2}\pa \Gamma^{\alpha}w-\hat{\omega}\op{W}^{(\alpha)}| \\
		&\quad +Ct^{-1}\bigl( r^{1/2}|\pa w|+|\op{W}|\bigr)\bigl(r^{1/2}|\pa \Gamma^{\alpha}w|+|\op{W}^{(\alpha)}|\bigr)|r^{1/2}\pa w-\hat{\omega}\op{W}| \\
		&\le Ct^{-1/2}\jb{t-r}^{-2}[w]^{3}_{s+1}. 
	\end{align*}
	Other three terms $R^{\alpha}_{2}$, $R^{\alpha}_{3}$, and $R^{\alpha}_{4}$ can be treated in the same way as $R_2$, $R_3$ and $R_4$, respectively.
	Therefore we obtain \eqref{eq;Halphaest} as desired.
\end{proof}

%%%%%%%%%%%%%%%%%%%%%%%%%%%%%%%%%%%%%%%%%%%%%%%%%%%%%%%%%%%%%%%%%%%%%%%%%%%%%%%%%%%%%%%%%%%%%%%%%%%
%%%%%%%%%%%%%%%%%%%%%%%%%%%%%%%%%%%%%%%%%%%%%%%%%%%%%%%%%%%%%%%%%%%%%%%%%%%%%%%%%%%%%%%%%%%%%%%%%%%

\section{Transformation for the Klein-Gordon components} \label{section;transform}
Among the nonlinear terms in $F_{\K}$ for the Klein-Gordon components, $F^{\w}_{\K}$ has the slowest decay.
To treat it in the decay estimate for the Klein-Gordon components, we use the transformation in Tsutsumi~\cite{Tsu}.
The idea of this transformation can go back to Kosecki~\cite{Kos}.
We would like to summarize the argument here.

For this purpose, we introduce
\begin{align}
               Q_{0}(\phi, \psi) := (\pa_{t}\phi)(\pa_{t}\psi)-(\nabla_{x}\phi) \cdot (\nabla_{x}\psi),
               \label{eq;Q0}
\end{align}
which is one of the \textit{null forms} introduced in \cite{Kla2} to characterize the null condition.
A key feature of the null forms is their faster decay.
Indeed, we have the following:

\begin{lem} \label{lem;null forms}
	For any non-negative integer $s$, there is a positive constant $C$ such that 
	\begin{align*}
		|Q_0(\zeta_j, \zeta_k)(t,x)|_s &\le C \jb{t+|x|}^{-1} \bigl([\zeta(t,x)]_{[s/2]+1}|\pa \zeta(t,x)|_s \\
		&\qquad \qquad \qquad \qquad \qquad +|\pa \zeta(t,x)|_{[s/2]}[\zeta(t,x)]_{s+1} \bigr)
	\end{align*} 
	for $t>0$, $x \in \R^2$, a smooth function $\zeta=(\zeta_j)_{1 \le j \le N}$ and $1 \le j,k \le N$.
\end{lem}

\begin{proof}
	This estimate in three space dimensions was proved in \cite{Kata3}, and we only give an outline here.
	Because $\Gamma^{\alpha}Q_0(\zeta_j, \zeta_k)$ can be written as a linear combination of $Q_0(\Gamma^{\beta}\zeta_j, \Gamma^{\gamma}\zeta_k)$ with $|\beta|+|\gamma| \le |\alpha|$ (see \cite{Kata5} for instance), it suffices to prove
	\begin{align*}
		|Q_0(\phi, \psi)| \le \jb{t+r}^{-1}\bigl([\phi]_1|\pa \psi|+|\pa \phi|[\psi]_1\bigr).
	\end{align*}
	If $0 \le t \le 2$ or $r \le t/2$, then we have $\jb{t+r} \le C\jb{t-r}$.
	Hence we get
	\begin{align*}
		|Q_0(\phi, \psi)| \le C|\pa \phi||\pa\psi| \le C\jb{t+r}^{-1}\jb{t-r}|\pa\phi||\pa\psi| \le C\jb{t+r}^{-1}[\phi]_1|\pa \psi|.
	\end{align*}
	On the other hand, when $1 \le t/2 \le r$, \eqref{rmk;profile} implies
	\begin{align*}
		|(\pa_a \phi)(\pa_b \psi)-(\omega_a D_- \phi)(\omega_b D_- \psi)| \le C\jb{t+r}^{-1} \bigl([\phi]_1|\pa \psi|+|\pa \phi|[\psi]_1\bigr),
	\end{align*}
	which leads to the desired result for $Q_0$ because $(\omega_0^2-\omega_1^2-\omega_2^2)(D_- \phi)(D_- \psi)=0$.
\end{proof}

Now we return our attention to the transformation.
For $N_0+1 \le k,l,m \le N$, direct calculations yield
\begin{align*}
	&\Box \bigl( (\pa_aw_k)(\pa_bw_l)(\pa_cw_m) \bigr) \\
	&\quad =(\pa_a \Box w_k)(\pa_b w_l)(\pa_c w_m)+(\pa_a w_k)(\pa_b \Box w_l)(\pa_c w_m)+(\pa_a w_k)(\pa_b w_l)(\pa_c \Box w_m)\\
	&\qquad +2(\pa_aw_k)Q_0(\pa_bw_l, \pa_cw_m)+2(\pa_bw_l)Q_0(\pa_cw_m, \pa_aw_k)\\
	&\qquad +2(\pa_cw_m)Q_0(\pa_aw_k, \pa_bw_l).
\end{align*}
Using that $\pa_a \Box w_k=\pa_a \bigl(F_k(v, \pa u) \bigr)$, and also applying Lemma \ref{lem;null forms} to estimate $Q_0$, we obtain
\begin{align}
	\notag
	\ab{\Box \bigl( (\pa_aw_k)(\pa_bw_l)(\pa_cw_m) \bigr)}_s &\le C|(v, \pa u)|^{4}_{[(s+1)/2]} |(v, \pa u)|_{s+1} \\
	&\qquad +C\jb{t+r}^{-1}|\pa w|_{[s/2]}[w]_{s+1}|\pa w|_s
	\label{eq;dadada}
\end{align}
for non-negative integer $s$.

\begin{lem} \label{lem;vtilder}
	For $1 \le j \le N_0$, we put 
	\begin{align}
		\tl{v}_j:=v_j-m_j^{-2}F_j^{\w}(\pa w).
		\label{eq;vtl}
	\end{align}
	Then, for any non-negative integer $s$, we have 
	\begin{align*}
		\ab{(\Box+m_j^2)\tl{v}_j-\bigl(F_j(v, \pa u)-F_{j}^{\w}(\pa w) \bigr)}_s &\le C|(v, \pa u)|^4_{[(s+1)/2]}|(v, \pa u)|_{s+1} \\
		&\quad +C\jb{t+r}^{-1}|\pa w|_{[s/2]}[w]_{s+1}|\pa w|_{s}.
	\end{align*}
\end{lem}

\begin{proof}
	As the definition of $\tl{v}_j$ yields
	\begin{align*}
		(\Box+m_j^2)\tl{v}_j=\bigl( F_j(v, \pa u)-F^{\w}_{j}(\pa w) \bigr)-m_j^{-2}\Box \bigl( F_j^{\w} \bigr),
	\end{align*}
	this lemma is an immediate consequence of \eqref{eq;dadada}.
\end{proof}

%%%%%%%%%%%%%%%%%%%%%%%%%%%%%%%%%%%%%%%%%%%%%%%%%%%%%%%%%%%%%%%%%%%%%%%%%%%%%%%%%%%%%%%%%%%%%%%%%%%
%%%%%%%%%%%%%%%%%%%%%%%%%%%%%%%%%%%%%%%%%%%%%%%%%%%%%%%%%%%%%%%%%%%%%%%%%%%%%%%%%%%%%%%%%%%%%%%%%%%

\section{Proof of the global existence} \label{section;sdge} 
Finally, we are in a position to prove Theorem \ref{thm;main}.
For a smooth solution $u=(v, w)$ to the Cauchy problem \eqref{eq;nkgw}-\eqref{eq;id} on $[0, T) \times \R^2$, we define
\begin{align}
	\begin{array}{ll}
		\op{E}[u](T) &:= \displaystyle \sup_{(t,x) \in [0,T) \times \R^2} \Bigl(\jb{t+r}|v(t,x)|_{I+1} \\
		&\qquad \qquad \qquad  \qquad \ +\jb{r}^{1/2}\jb{t-r}^{1-\rho}|\pa w(t,x)| \\
		&\qquad \qquad \qquad  \qquad \  +\jb{t+r}^{-\kappa}\jb{r}^{1/2}\jb{t-r}^{1-\rho}|\pa w(t,x)|_{I} \Bigr),
	\end{array} 
\label{eq;apriori}
\end{align}
where $\rho$ and $\kappa$ are small positive constants to be fixed later, and $I$ is a fixed positive integer with $I \ge 9$.
$\kappa$ is assumed to be so small compared to $\rho$ that we have $8\kappa < \rho$ at least.

For the proof of Theorem \ref{thm;main}, it suffices to show the following: 
If we choose appropriate $\rho$ and $\kappa$, then for any large number $M$, there is a positive number $\eps_0=\eps_0(M)$, being independent of $T$, such that $\op{E}[u](T) \le M\eps$ implies $\op{E}[u](T) \le M\eps/2$ for $\eps \in (0, \eps_0]$.
Indeed, this property and the bootstrap argument lead to \textit{a priori} estimates of $u$ for sufficiently small $\eps$, and we have the global existence.

We assume $\op{E}[u](T) \le M\eps$ from now on.
In the following arguments, $M (\ge 1)$ is sufficiently large, and $\eps$ is supposed to be so small that $M\eps \le M^2 \eps \le 1$.
Since $f,g \in C^{\infty}_{0}$, we may assume \eqref{eq;idfinite} with some $R > 0$, and consequently we have \eqref{eq;finite}.

Throughout this proof, $C$ denotes a positive constant which is independent of $M$, $\eps$ and $T$.

Recall the definition $W_-(t,r)=\min\{\jb{r}, \jb{t-r}\}$.
Then we have 
\begin{align}
	\label{eq;vdu}
	\ab{\bigl(v(t,x), \pa u(t,x)\bigr)} &\le CM\eps\jb{t+r}^{-1/2}W_-(t,r)^{-1/2}, \\
	\label{eq;vduI}
	\ab{\bigl(v(t,x), \pa u(t,x)\bigr)}_{I} &\le CM\eps\jb{t+r}^{\kappa-1/2}W_-(t,r)^{-(1/2)-\kappa}
\end{align}
for $(t,x) \in [0,T) \times \R^2$, since there is a positive constant $C$ such that
\begin{align*}
	C^{-1}\jb{t+r}W_-(t, r) \le \jb{r}\jb{t-r} \le C\jb{t+r}W_-(t, r)
\end{align*} 
for $(t,x) \in [0,T) \times \R^2$.

{\bf Step 1: \ Energy estimates for wave and Klein-Gordon components.}
Recall the definitions of $|\phi(t,x)|_s$ and $\| \phi(t)\|_s$ for a smooth function $\phi$ and a non-negative integer $s$.
For simplicity of exposition, we also set $|\phi(t,x)|_{-1}$ and $\| \phi(t) \|_{-1}$ to be zero in the sequel.
Let $0 \le l \le 2I+1$.
By the Leibniz formula, \eqref{eq;vdu} and \eqref{eq;vduI}, we obtain
\begin{align}
	\notag
	|F(v, \pa u)|_l &\le C\Bigl(|(v, \pa u)|^2|(v, \pa u)|_l+|(v, \pa u)|_{[l/2]}^2|(v, \pa u)|_{l-1}\Bigr) \\
	\notag
	&\quad +C|(v, \pa u)|_{[l/2]}^3|(v, \pa u)|_{l} \\
	&\le CM^2\eps^2 \left((1+t)^{-1}|(v, \pa u)|_l+(1+t)^{2\kappa-1}|(v, \pa u)|_{l-1} \right),
	\label{eq;Flest}
\end{align}
since we have $[l/2] \le I$ for $0 \le l \le 2I+1$.
Because of \eqref{eq;Cs}, applying the standard energy inequality for wave and Klein-Gordon equations to \eqref{eq;comm}, we get
\begin{align} 
	\notag
	\|(v, \pa u)(t)\|_l &\le C\left( \|(v,\pa u)(0)\|_l+ \int^{t}_{0}\left\| F \bigl((v, \pa u)(\tau)\bigr)\right\|_l \, d\tau \right) \\
	\notag
	&\le C\eps +CM^2\eps^2\int^{t}_{0}(1+\tau)^{-1}\|(v, \pa u) \|_l \, d\tau \\
	&\quad +CM^2\eps^2\int^{t}_{0}(1+\tau)^{2\kappa-1}\|(v, \pa u)(\tau)\|_{l-1} \, d\tau,
	\label{eq;sdene}
\end{align}
where the constant $C$ can be chosen independently of $l$.

It follows from \eqref{eq;sdene} with $l=0$ and the Gronwall lemma that
\begin{align*}
	\|(v, \pa u)(t)\|_0 \le C\eps(1+t)^{CM^2\eps^2}.
\end{align*}
Similarly, applying the Gronwall lemma to \eqref{eq;sdene}, we can inductively show
\begin{align}
	\|(v, \pa u)(t)\|_l \le C\eps(1+t)^{CM^2\eps^2+2l\kappa}
	\label{eq;ene}
\end{align}
for $0\le l \le 2I+1$.
Especially, we have 
\begin{align}
	\jb{t}^{2\kappa}\|(v, \pa u)(t)\|_{2I}+\|(v, \pa u)(t)\|_{2I+1} \le C\eps\jb{t}^{\theta},
	\label{eq;t2k}
\end{align}
where
\begin{align*}
	\theta:=C_{\ast}M^2\eps^2+2(2I+1)\kappa \le C_{\ast}\eps+2(2I+1)\kappa.
\end{align*}
Here, we write $C_{\ast}$ for the specific positive constant $C$ appeared in \eqref{eq;ene}.
Given $\rho$, we choose small $\kappa$ and $\eps_1$ such that we have $\theta \le \rho/2$ for $0 < \eps \le \eps_1$.
Then \eqref{eq;t2k} can be written as 
\begin{align}
	\jb{t}^{2\kappa}\|(v, \pa u)(t)\|_{2I}+\|(v, \pa u)(t)\|_{2I+1} \le C\eps\jb{t}^{\rho/2}
	\label{eq;Klasob}
\end{align}
for small $\eps$.

{\bf Step 2: \ Rough decay estimates for wave and Klein-Gordon components.}
By \eqref{eq;vduI} and \eqref{eq;Klasob}, we obtain
\begin{align*} 
	||F(v, \pa u)(t)||_{2I+1} \le CM^2\eps^2\jb{t}^{(\rho/2) -1}.
\end{align*}  
It follows from Lemma \ref{lem;v} that
\begin{align}
	\jb{t+r}|v(t,x)|_{2I-3} \le C\eps+CM^2\eps^2\jb{t}^{\rho/2} \le C\eps\jb{t}^{\rho/2}.
	\label{eq;vest1}
\end{align}

By \eqref{eq;vduI} and \eqref{eq;Klasob}, together with Lemma \ref{lem;xphix}, we get
\begin{align}
	\notag
	\jb{r}^{1/2}|F_{\W}(v, \pa u)|_{2I-1} &\le C\jb{r}^{1/2}|(v, \pa u)|^2_I|(v, \pa u)|_{2I-1} \\
	\notag
	&\le C|(v, \pa u)|^2_I \|(v, \pa u)\|_{2I+1} \\
	\notag
	&\le C M^2 \eps^3 \jb{t+r}^{(\rho/2)+2\kappa-1} W_-(t, r)^{-2\kappa-1} \\
	\label{eq;rFW}
	&\le C M^2 \eps^3 \jb{t+r}^{\rho-2\kappa-1} W_-(t, r)^{-2\kappa-1},
\end{align}
since $8\kappa < \rho$.
As \eqref{eq;rFW} implies
\begin{align*}
	\sum^1_{s=0} \sum_{|\alpha| \le 2I-1-s} \op{B}_{2\kappa-\rho, 2\kappa, s}[\Gamma^{\alpha}F_{\W}](t,x) \le CM^2\eps^3 \le C\eps,
\end{align*}
it follows from Lemma \ref{lem;hom} and Lemma \ref{lem;inhom} with $\xi=\rho$ and $\zeta=\eta=\kappa$ that 
\begin{align}
	\label{eq;w1}
	|w(t,x)|_{2I-1} &\le C\eps\jb{t+r}^{\rho-(1/2)}\jb{t-r}^{-\kappa}, \\
	\label{eq;dw1}
	|\pa w(t,x)|_{2I-2} &\le C\eps\jb{t+r}^{\rho}\jb{r}^{-1/2}\jb{t-r}^{-1-\kappa}.
\end{align}

Note that \eqref{eq;vest1} and \eqref{eq;dw1} yield
\begin{align}
	|(v, \pa u)(t,x)|_{2I-4} \le C\eps\jb{t+r}^{\rho}\jb{r}^{-1/2}\jb{t-r}^{-1/2-\kappa},
	\label{eq;vdu2I-4}
\end{align}
since $\jb{t+r}^{-(1/2)-(\rho/2)} \le C\jb{t-r}^{-(1/2)-\kappa}$.

{\bf Step 3: \ Better decay estimates for the wave components.}
Recall the definition \eqref{dfn;light cone} of $\Lambda_{T,R}$.
We set $\Lambda^{\rm c}_{T,R}=\bigl( [0,T) \times \R^2 \bigr) \backslash \Lambda_{T,R}$.
If $r>t+R$, we have $\pa w(t,x)=0$.
On the other hand, when $t<2$ or $t > 2r$, we have $\jb{t+r} \le C\jb{t-r}$ with some universal positive constant $C$.
Therefore \eqref{eq;dw1} leads to 
\begin{align*}
	|\pa w(t,x)|_{2I-2} \le C\eps\jb{r}^{-1/2}\jb{t-r}^{(\rho-\kappa)-1}, \ \ (t,x) \in \Lambda^{\rm c}_{T,R}.
\end{align*}
Therefore we get 
\begin{align}
	\sup_{(t,x) \in \Lambda^{\rm c}_{T,R}} \jb{r}^{1/2}\jb{t-r}^{1-\rho}|\pa w(t,x)|_{2I-2} \le C\eps.
	\label{eq;fwestOUT}
\end{align}

Hence our task is to estimate $\pa w$ in $\Lambda_{T,R}$.
Let $(t,x) \in \Lambda_{T,R}$ in the rest of this step.
Recall the definitions of $\op{W}$, $\op{W}^{(\alpha)}$, $H$ and $H_{\alpha}$, as well as $\op{W}^{\ast}_{0}$ and $\op{W}^{(\alpha)\ast}_{0}$ in Section \ref{section;profile}.

Let $|\alpha|=s \le 2I-4$.
We start with the estimate for $\op{W}^{(\alpha)\ast}_0$, which includes that for $\op{W}^{\ast}_{0}$ as the special case of $|\alpha|=0$.
By \eqref{eq;w1} and \eqref{eq;dw1}, we get
\begin{align*}
	|\op{W}^{(\alpha)}(t,x)| \le C\eps t^{\rho}\jb{t-r}^{-1-\kappa}.
\end{align*}
Recalling that $t_0(\sigma)$ is equivalent to $\jb{\sigma}$ for $\sigma=r-t$ in $\Lambda_{T,R}$ (see \eqref{eq;t0sigma}), we obtain
\begin{align}
	|\op{W}^{(\alpha)\ast}_{0}(\sigma, \omega)| \le C\eps \jb{\sigma}^{\rho-1-\kappa} \le C\eps \jb{\sigma}^{\rho-1}.
	\label{eq;Walphaest1}
\end{align}
It follows from \eqref{eq;w1} and \eqref{eq;dw1} that
\begin{align}
	[w(t,x)]_{2I-1} =|w(t,x)|_{2I-1}+\jb{t-r}|\pa w(t,x)|_{2I-2} \le C\eps t^{\rho-(1/2)}\jb{t-r}^{-\kappa}.	
	\label{eq;sharp1}
\end{align}

We use Lemma \ref{lem;Hest} to estimate $H$ and $H_\alpha$.
It follows from \eqref{eq;sharp1}, \eqref{eq;w1}, \eqref{eq;vduI}, \eqref{eq;vest1} and \eqref{eq;vdu2I-4} (in the order of application) that
\begin{align}
	\notag
	|H| &\le C\eps^3 t^{3\rho-2}\jb{t-r}^{-2-3\kappa}+C\eps t^{\rho-2}\jb{t-r}^{-\kappa} \\
	\notag
	&\quad +CM^2\eps^3 t^{2\kappa-(1/2)}\jb{t-r}^{-1-2\kappa} \bigl(t^{(\rho/2)-1}+t^{2\rho-1}\jb{t-r}^{-1-2\kappa}\bigr) \\
	\label{eq;Hest1}
	&\le C\eps t^{3\rho-(3/2)} \jb{t-r}^{-2\rho-\kappa-(1/2)},
\end{align} 
provided that $\rho$ and $\kappa$ are chosen to be sufficiently small.
Just in the same manner, we get
\begin{align}
	|H_\alpha| \le C t^{1/2}|\pa w|^{3}_{s-1}+C\eps t^{3\rho-(3/2)} \jb{t-r}^{-2\rho-\kappa-(1/2)}.
	\label{eq;Halphaest1}
\end{align}

Using \eqref{eq;Walphaest1}, \eqref{eq;sharp1} and \eqref{eq;Hest1}, we obtain from Lemma \ref{lem;rdw} that
\begin{align*}
	\jb{r}^{1/2}|\pa w(t,x)| &\le C\eps\jb{t-r}^{\rho-1}+C\eps\jb{t-r}^{-2\rho-\kappa-(1/2)}\int^{\infty}_{t_0(\sigma)}\tau^{3\rho-(3/2)} \ d\tau \\
	&\quad +C\eps t^{\rho-1}\jb{t-r}^{-\kappa} \le C\eps\jb{t-r}^{\rho-1}.
\end{align*}
In other words, we have proved
\begin{align}
	\sup_{(t,x) \in \Lambda_{T,R}} \jb{r}^{1/2}\jb{t-r}^{1-\rho}|\pa w(t,x)| \le C\eps.
	\label{eq;fdwest}
\end{align}

From \eqref{eq;w1} and \eqref{eq;fdwest}, we get $|\op{W}(t,x)| \le C\eps$, which implies
\begin{align*}
 \no{\op{G} \left( \omega, \op{W}(t,x) \right)}_{\Op} \le 2C_1 \eps^2
\end{align*}
with some positive constant $C_1$.
Lemma \ref{lem;dGw} together with \eqref{eq;Walphaest1}, \eqref{eq;sharp1} and \eqref{eq;Halphaest1} implies
\begin{align*}
	&\jb{r}^{1/2}|\pa w(t,x)|_s\\
	&\qquad \le C\eps t^{C_1\eps^2}\jb{\sigma}^{\rho-1}+C\eps t^{C_1\eps^2}\jb{\sigma}^{-2\rho-\kappa-(1/2)}\int^{\infty}_{t_0(\sigma)}\tau^{3\rho-(3/2)-C_1\eps^2} \ d\tau \\
	&\qquad \quad +Ct^{C_1\eps^2}\int^{t}_{t_{0}(\sigma)}\tau^{(1/2)-C_1\eps^2}\ab{\pa u \left(\tau, (\sigma+\tau)\omega \right)}^{3}_{s-1} \ d\tau+C\eps t^{\rho-1}\jb{\sigma}^{-\kappa} \\
	&\qquad \le C\eps t^{C_1\eps^2}\jb{\sigma}^{\rho-1} \\
	&\qquad \quad +Ct^{C_1\eps^2}\int^{t}_{t_0(\sigma)}\tau^{-C_1\eps^2-1} \left( \jb{\sigma+\tau}^{1/2}\ab{\pa u \left( \tau, (\sigma+\tau)\omega \right)}_{s-1}\right)^3 \ d\tau,
\end{align*}
where $\sigma=r-t$.
Using \eqref{eq;fdwest}, we obtain
\begin{align*}
	\jb{r}^{1/2}|\pa w(t,x)|_1 &\le C\eps t^{C_1\eps^2}\jb{\sigma}^{\rho-1}+C\eps^3 t^{C_1\eps^2}\jb{\sigma}^{3(\rho-1)}\int^{t}_{t_0(\sigma)}\tau^{-C_1\eps^2-1} \ d\tau\\
	&\le C\eps t^{C_1\eps^2}\jb{\sigma}^{\rho-1}.
\end{align*}
Similarly, by induction, we obtain
\begin{align*}
	\jb{r}^{1/2}|\pa w(t, x)|_s \le C\eps t^{3^{s-1} C_1\eps^2}\jb{\sigma}^{\rho-1}
\end{align*}
for $1 \le I \le 2I-4$.
If $\eps$ is sufficiently small to satisfy
\begin{align*}
	3^{2I-5}C_1\eps^2 \le \kappa,
\end{align*}
we obtain
\begin{align}
	\sup_{(t,x) \in \Lambda_{T,R}} t^{-\kappa}\jb{r}^{1/2}\jb{t-r}^{1-\rho}|\pa w(t,x)|_{2I-4} \le C\eps.
	\label{eq;fdwpro}
\end{align}

{\bf Step 4: \ Better decay estimates for the Klein-Gordon components.}
We will make use of Lemma \ref{lem;vtilder} to improve decay estimates for the Klein-Gordon components.
Let $(t,x) \in [0,T) \times \R^2$.
For $1 \le j \le N_0$, we put
\begin{align*}
	\tl{v}_j=v_j-m^{-2}_{j}F^{\w}_{j}(\pa w).
\end{align*}

From \eqref{eq;vduI}, \eqref{eq;Klasob}, \eqref{eq;vest1} and \eqref{eq;vdu2I-4}, we get
\begin{align*}
	\no{F^{\kl}_{\K}}_{2I-4} &\le C\no{|v|_{I-1}}^{2}_{L^{\infty}} \no{v}_{2I-3} \le CM^{2}\eps^{3}\jb{t}^{(\rho/2)-2}, \\
	\no{F^{\kkw}_{\K}+F^{\kww}_{\K}}_{2I-4} &\le C\no{|v|_{2I-3}|(v, \pa u)|_{2I-4}}_{L^{\infty}}\no{\pa w}_{2I-4} \\
	&\le C\eps^3\jb{t}^{2\rho-(3/2)}, \\
	\no{F^{\h}_{\K}}_{2I-4} &\le C\no{|(v, \pa u)|_{I-2}}^{3}_{L^{\infty}} \no{(v, \pa u)}_{2I-4} \\
	&\le CM^{3}\eps^{4}\jb{t}^{(\rho/2)+3\kappa-(3/2)}.
\end{align*}
Similarly to the last line, we also have
\begin{align*}
	\no{|(v, \pa u)|^{4}_{[(2I-3)/2]}|(v, \pa u)|_{2I-3}}_{L^{2}} \le CM^4 \eps^{5}\jb{t}^{(\rho/2)+4\kappa-2}.
\end{align*}

By \eqref{eq;w1} and \eqref{eq;dw1}, we get
\begin{align*}
	[w(t,x)]_{2I-3} \le C\eps\jb{t+r}^{\rho}\jb{r}^{-1/2}\jb{t-r}^{-\kappa} \le C\eps\jb{t+r}^{\rho-\kappa}W_{-}(t,r)^{-1/2}.
\end{align*}
Accordingly we obtain
\begin{align*}
	&\no{\jb{t+|\cdot|}^{-1} |\pa w|_{I-2}[w]_{2I-3}|\pa w|_{2I-4}}_{L^{2}} \\
	&\quad \le C\no{\jb{t+|\cdot|}^{-1}|\pa w|_{I-2}[w]_{2I-3}}_{L^{\infty}}\no{\pa w}_{2I-4} \\
	&\quad \le CM\eps^3 \jb{t}^{(3\rho/2)-(3/2)}.
\end{align*}
Gathering the above estimates, we see from Lemma \ref{lem;vtilder} that
\begin{align}
	\no{(\Box+m^{2}_{j})\tl{v}_{j}}_{2I-4} \le CM^2\eps^3 \jb{t}^{-1-\kappa} \le C\eps\jb{t}^{-1-\kappa}.
	\label{eq;vtilder}
\end{align}
Lemma \ref{lem;v} implies
\begin{align*}
	\jb{t+r}|\tl{v}(t,x)|_{2I-8} \le C\eps.
\end{align*}
By \eqref{eq;vduI} and \eqref{eq;dw1}, we have
\begin{align*}
	\ab{v(t,x)-\tl{v}(t,x)}_{2I-8} &\le C\left|F^{\w}_{\K}(\pa w)(t,x)\right| \le CM^2\eps^3\jb{t+r}^{2\kappa+\rho-(3/2)} \\
	&\le C\eps\jb{t+r}^{-1}.
\end{align*}
Therefore we obtain 
\begin{align}
	\jb{t+r}|v(t,x)|_{2I-8} \le C\eps.
	\label{eq;fvest} 
\end{align}

{\bf The final step.}
Because $I+1 \le 2I-8$ and $I \le 2I-4$ for $I \ge 9$, from \eqref{eq;fwestOUT}, \eqref{eq;fdwest}, \eqref{eq;fdwpro} and \eqref{eq;fvest}, we find that there is a positive constant $C_0$ and $\eps_0=\eps_0(M)$ such that we have
\begin{align*}
	\op{E}[u](T) \le C_0\eps
\end{align*}
for $\eps \in (0, \eps_0]$.
If $M$ is sufficiently large to satisfy $M \ge 2C_0$, then we have $\op{E}[u](T) \le M\eps/2$ for $\eps \in (0,\eps_0]$, as desired.
This completes the proof.

%%%%%%%%%%%%%%%%%%%%%%%%%%%%%%%%%%%%%%%%%%%%%%%%%%%%%%%%%%%%%%%%%%%%%%%%%%%%%%%%%%%%%%%%%%%%%%%%%%%
%%%%%%%%%%%%%%%%%%%%%%%%%%%%%%%%%%%%%%%%%%%%%%%%%%%%%%%%%%%%%%%%%%%%%%%%%%%%%%%%%%%%%%%%%%%%%%%%%%%

\section{Remarks on the asymptotic behavior of global solutions} \label{section;AB}

In this section, we briefly discuss the asymptotic behavior of global solutions.
Note that all the estimates in the previous section hold with $T=\infty$.

Firstly, we consider the Klein-Gordon components $v=(v_j)_{1 \le j \le N_0}$.
We say that $v$ is \textit{asymptotically free}, if there is $(\varphi^{+}_{j}, \psi^{+}_{j}) \in H^1(\R^2) \times L^2(\R^2)$ for $1 \le j \le N_0$ such that 
\begin{align*}
	\lim_{t \rightarrow \infty} \left( \no{v(t)-\vp(t)}_{H^1(\R^2)}+\no{\pa_t v(t)-\pa_t \vp(t)}_{L^2(\R^2)} \right)=0,
\end{align*} 
where $\vp=(\vp_j)_{1 \le j \le N_0}$ satisfies $(\Box+m^2_j)\vp_j=0$ with $(\vp_j, \pa_t \vp_j)(0)=(\varphi^+_j, \psi^+_j)$ for $1 \le j \le N_0$.
We refer to $(\varphi^{+}, \psi^{+})=\bigl( (\varphi^+_j), (\psi^+_j) \bigr)$ as the \textit{asymptotic data}.

By the standard argument, \eqref{eq;vtilder} shows that $\tl{v}$ is asymptotically free.
Since \eqref{eq;vduI} and \eqref{eq;Klasob} leads to 
\begin{align*}
	\no{v(t)-\tl{v}(t)}_1 \le C \no{F^{\w}_{\K} \bigl( \pa w(t) \bigr)}_1 \le C\eps^3 \jb{t}^{2\kappa+(\rho/2)-1} \rightarrow 0, \ \ t \rightarrow \infty, 
\end{align*}
we see that $v$ also is asymptotically free.
Moreover, the asymptotic data $(\varphi^+, \psi^+)$ satisfies $(\varphi^+, \psi^+)= \eps(f, g)+\op{O}(\eps^3)$ as $\eps \rightarrow +0$ in $H^1 \times L^2$.
In other words, the asymptotic data is close to the original data for small $\eps$.

Secondly, let us consider the wave components $w=(w_j)_{N_0+1 \le j \le N}$.
For systems of semilinear wave equations (namely for the case $N_0=0$), there is a wide variety of asymptotic behavior for global solutions under the KMS condition, depending on the nonlinearity: some solutions are asymptotically free, and others are not.
Even if the solution is asymptotically free, the asymptotic data can be away from the original data, and decay of the energy may also occur for some nonlinearity.
For these results, see \cite{Kata5,KMatsuS}, Katayama-Murotani-Sunagawa~\cite{KMuroS}, Nishii-Sunagawa~\cite{NS}, Nishii-Sunagawa-Terashita~\cite{NST}.
In these works, the main tool to obtain the asymptotic behavior is the profile system \eqref{eq;profile}.

For our system, we obtain from \eqref{eq;Hest1} that 
\begin{align}
	\pa_+ \op{W}(t,x) =-\frac{1}{2t}F^{\red}_{\W}\bigl(\omega, \op{W}(t,x)\bigr)+\op{O}\bigl(t^{3\rho-(3/2)}\jb{t-r}^{-2\rho-\kappa-(1/2)}\bigr)
	\label{eq;profileab}
\end{align}
for $(t,x) \in \Lambda_{R}$.
By the a \textit{priori} estimate, we also have
\begin{align}
	|\pa w(t,x)| \le C\eps\jb{r}^{-1/2}\jb{t-r}^{\rho-1}, \ \ \ (t,x) \in [0, \infty) \times \R^2.
	\label{eq;westab}
\end{align}
Using these estimates, we can recover most of the asymptotic behavior results in the above works with some non-essential modification.
Therefore, the asymptotic behavior of $w$ is governed only by $F^{\w}_{\W}$ through $F^{\red}_{\W}$, and other parts of the nonlinearity do not affect the asymptotic behavior of $w$.

To illustrate it, let us consider the system \eqref{eq;ex} with $a=0$ and $c=1$ in Example \ref{ex}.
For single wave equation $\Box w=-(\pa_t w)^3$, it is known that the energy decay occurs (see \cite{Kata5,KMatsuS,KMuroS}).
We show that the same is true for the wave component $w$ in \eqref{eq;ex}.
Since $F^{\red}_{\W}(\omega, Y)=Y^3$, we can apply Lemma 10.25 in \cite{Kata5} to our profile system \eqref{eq;profileab}, and we see that there is $A_0=A_0(\sigma, \omega)$ such that
\begin{align*}
	\op{W}(t,x)=A(t,r-t,\omega)+\op{O}\bigl(\eps t^{4\rho-(1/2)}\jb{t-r}^{-3\rho-\kappa-(1/2)}\bigr)
\end{align*}
for $(t,x) \in \Lambda_{R}$ and $|A_0(\sigma, \omega)| \le C\eps\jb{\sigma}^{-\rho-\kappa-1}$, where $A=A(t, \sigma, \omega)$ is a solution to $\pa_t A=-A^3/(2t)$ with $A(1, \sigma, \omega)=A_0(\sigma, \omega)$.
It follows from Lemma \ref{lem;profile} and \eqref{eq;sharp1} that
\begin{align*}
	r^{1/2}\pa w(t,x)=\hat{\omega}A(t, r-t, \omega)+\op{O}\bigl(\eps t^{4\rho-(1/2)}\jb{t-r}^{-3\rho-\kappa-(1/2)}\bigr), \ (t,x) \in \Lambda_{R}.
\end{align*}
$A$ can be solved explicitly (see \cite{Kata5} for example), and we can show
\begin{align*}
	|A(t, \sigma, \omega)| \le C(\log t)^{-1/2},
\end{align*}
which, together with \eqref{eq;westab}, yields
\begin{align*}
	r^{1/2}|\pa w(t,x)| \le C\eps \min\{(\eps^2\log t)^{-1/2}, \jb{t-r}^{\rho-1}\}.
\end{align*}
In this way, the enhanced decay estimate for $w$ in \cite{Kata5} is recovered, and we can follow the proof of Theorem 10.26 in \cite{Kata5} to obtain
\begin{align*}
	\no{\pa w(t)}_{L^2} \le \frac{C\eps}{(\eps^2\log t)^{(1-\mu)/4}} \rightarrow 0, \ \ t \rightarrow \infty
\end{align*}
with $\mu=\rho/(1-\rho)$.

\ \\
\ \\

\end{document}